\documentclass[12pt,a4paper]{amsart}
\usepackage[margin=2.5cm]{geometry}
\usepackage{amsmath,amsthm}
\newtheorem{theorem}{Theorem}[section]

\newtheorem{lemma}[theorem]{Lemma}
\newtheorem{proposition}[theorem]{Proposition}
\theoremstyle{definition}
\newtheorem{definition}[theorem]{Definition}

\theoremstyle{remark}
\newtheorem{remark}[theorem]{Remark}
\numberwithin{equation}{section}
\usepackage[T1]{fontenc}
\usepackage[utf8]{inputenc}
\usepackage{lmodern}
\usepackage{amssymb}
\usepackage[hidelinks,bookmarks=false]{hyperref}
\date{}

\begin{document}
\title[Jordan Diophantine sets of matrices]{Cardinalities of Jordan Diophantine sets of upper triangular integer matrices}
\author[Z. Franu\v{s}i\'{c} and T. Pejkovi\'{c}]{Zrinka Franu\v{s}i\'{c} and Tomislav Pejkovi\'{c}}

\address[Zrinka Franušić]{Department of Mathematics, Faculty of Science\\University of Zagreb\\Bijenička cesta 30, 10000 Zagreb, Croatia}
\email[Zrinka Franušić]{fran@math.hr}
\address[Tomislav Pejković]{Department of Mathematics, Faculty of Science\\University of Zagreb\\Bijenička cesta 30, 10000 Zagreb, Croatia}
\email[Tomislav Pejković]{pejkovic@math.hr}

\begin{abstract}Let $U$ be the ring of upper triangular $2\times2$ integer matrices. For $N\in U$, a Jordan $D(N)$-set is a set of distinct nonzero matrices in $U$ such that $(AB+BA)/2+N$ is a square in $U$ for any two distinct elements $A$ and $B$. We determine all $N$ for which an infinite Jordan $D(N)$-set exists. If no infinite set exists, every such set has at most six elements. If $N$ is not a difference of two squares in $U$, the bound is five. Both bounds are best possible. We also give necessary and sufficient conditions for the existence of infinite sets with nonzero pairwise Jordan products and of infinite sets consisting of nonsingular matrices.\end{abstract}
\maketitle

\section{Introduction}

Let $n$ be an integer. A set of distinct nonzero integers has the
\emph{$D(n)$-property} if the product of any two distinct elements,
increased by $n$, is a square. Such a set with $m$ elements is called a
\emph{$D(n)$-$m$-tuple} \cite{DujellaBook}. 
One of the main questions is to determine how large such a set can be.
In this paper, we study this question
for upper triangular $2\times2$ integer matrices.

Mouanda and Vincent \cite{MouandaVincent} studied Diophantine
tuples of matrices defined using the usual matrix product.
Dujella and Franušić \cite{DFsquares,DFJordan} considered the
Jordan product $A\circ B=(AB+BA)/2$ in place of $AB$.
Following their terminology, we use the following definition.

\begin{definition}
Let $U$ be the ring of upper triangular $2\times2$ integer matrices.
For $N\in U$, a set of distinct nonzero matrices in $U$ is called
a \emph{Jordan $D(N)$-set} if $A\circ B+N$ is a square in $U$
for any two distinct elements $A,B$ of the set.

A Jordan $D(N)$-set with $m$ elements is called a
\emph{Jordan $D(N)$-$m$-tuple}.

A Jordan $D(N)$-set is called \emph{nondegenerate} if
$A\circ B\ne0$ for any two distinct elements $A,B$ of the set.
\end{definition}

Note that if $N$ is not a square in $U$, every Jordan $D(N)$-set is nondegenerate.
Indeed, if two distinct members satisfied $A\circ B=0$, then
$N=A\circ B+N$ would be a square in $U$.

Dujella and Franušić studied the connection between the existence
of Jordan $D(N)$-quadruples and representations of $N$
as a difference of two squares in $U$.
In \cite{DFsquares}, they determined which matrices $N\in U$
can be represented as a difference of two squares in $U$.
They also constructed a Jordan $D(N)$-quintuple for a matrix
$N$ that is not a difference of two squares in $U$ and proved
that this quintuple cannot be extended to a Jordan
$D(N)$-sextuple \cite[Proposition~11]{DFJordan}. These existence results, together with nonexistence results
for several congruence classes, give a partial classification
of matrices $N\in U$ admitting a Jordan $D(N)$-quadruple,
with two congruence classes left unresolved
\cite[Section~4.5]{DFJordan}.
Dujella \cite[Problem~1.10]{DujellaProblems} reports that
Wildner proved existence in these two remaining classes.
This result is not included in the archived revision
\cite{WildnerArchive} considered here.

Wildner \cite[Proposition~4.3 and Theorem~4.5]{WildnerArchive}
constructed infinite Jordan $D(nI)$-sets of pairwise noncommuting
nonsingular matrices in $U$ whenever
$n\not\equiv2\pmod4$ and $n\notin\{1,4\}$, where $I$ is the
identity matrix.
His Proposition~5.1 gives such an infinite Jordan $D(N)$-set
whenever $N+a^2I$ is a square in $U$ for some nonzero integer $a$.

We determine all matrices $N\in U$ for which an infinite Jordan
$D(N)$-set exists (Theorem~\ref{thm:unrestricted-dichotomy}).
If no infinite set exists, every Jordan $D(N)$-set has at most
six elements, and this bound is best possible.
In Section~2, we establish the sharper bound of five when $N$
is not a difference of two squares in $U$, and show that it is
best possible.
Section~3 completes the classification by constructing infinite
sets from suitable representations $N=K^2-A^2$ and determining
the maximum cardinality in the exceptional cases.

We also study nondegenerate sets and sets consisting of
nonsingular matrices. For each class, we determine when infinite
sets exist and give upper bounds in the finite cases
(Sections~4 and~5). We construct sets attaining these bounds for
several families of parameters.
The proofs use the nonexistence of $D(1)$- and $D(4)$-quintuples
and of $D(-1)$- and $D(-4)$-quadruples in positive integers
\cite{BTF,BCM,HTZ}.

\section{A sharp bound of five}

Throughout the paper, we write $(x,y,z)$ for the matrix
$\left(\begin{smallmatrix}x&y\\0&z\end{smallmatrix}\right)$,
also allowing rational entries when needed.
For integers $x,y,z$, we have
$$ (x,y,z)^2=(x^2,y(x+z),z^2).$$
It follows that $(p,r,q)\in U$ is a square in $U$ if and only if
there are integers $x,z$ such that $p=x^2$, $q=z^2$, and
either $x+z\ne0$ and $x+z\mid r$, or $x+z=0$ and $r=0$.
Each diagonal square root may have either sign.
In particular, if the diagonal entries of a square have the same
parity, then its upper-right entry is even.

For $A_i=(a_i,b_i,d_i)$ and $A_j=(a_j,b_j,d_j)$ in $U$, we have
$$A_i\circ A_j
 =\left(a_i a_j,\,
 \frac{b_i(a_j+d_j)+b_j(a_i+d_i)}{2},\,
 d_i d_j\right).$$
Thus $A_i\circ A_j\in U$ if and only if
$b_i(a_j+d_j)+b_j(a_i+d_i)$ is even.

The following lemmas will be used to prove the upper bound.

\begin{lemma}\label{lem:parity-count}
Let $N=(p,r,q)\in U$, where $r$ is odd and $p\equiv q\pmod2$. 
Every Jordan \(D(N)\)-set contains at most three matrices whose diagonal entries have the same parity.
If the set contains a matrix whose diagonal entries have different parities, then at most one matrix has diagonal entries of the same parity.
\end{lemma}
\begin{proof}
Write
$A_i=(a_i,b_i,d_i)$
for the matrices in the given Jordan $D(N)$-set.
For each matrix with $a_i\equiv d_i\pmod2$, put
$e_i=(a_i+d_i)/2$.

For two distinct matrices $A_i,A_j$ with
$a_i\equiv d_i\pmod2$ and $a_j\equiv d_j\pmod2$, the diagonal
entries of $A_i\circ A_j+N$ are $a_i a_j+p$ and $d_i d_j+q$.
They have the same parity because $p\equiv q\pmod2$.
Since $A_i\circ A_j+N$ is a square in $U$, its upper-right
entry $b_i e_j+b_j e_i+r$ is even. As $r$ is odd, we obtain
\[
 b_i e_j+b_j e_i\equiv1\pmod2.
\]

To prove the first assertion, we may assume that there are
at least two matrices whose diagonal entries have the same
parity. For each of them, $b_i$ and $e_i$ cannot both be even,
since then $b_i e_j+b_j e_i$ would be even for any other
such matrix $A_j$.
Moreover, two distinct such matrices cannot satisfy both
$b_i\equiv b_j\pmod2$ and $e_i\equiv e_j\pmod2$, since this
would give
\[
 b_i e_j+b_j e_i\equiv2b_i e_i\equiv0\pmod2,
\]
contrary to the preceding congruence.
Thus $(b_i\bmod2,e_i\bmod2)$ must be one of
$(0,1)$, $(1,0)$ and $(1,1)$, and each pair can occur for
at most one matrix. This is the argument of \cite[Lemma~7]{DFJordan} and gives
the bound of three.

Now suppose that some matrix $A_k$ has diagonal entries
of different parities, so $a_k+d_k$ is odd.
Let $A_i$ be any matrix whose diagonal entries have the
same parity. Since $A_i\circ A_k+N$ and $N$ have integer
entries, the upper-right entry of $A_i\circ A_k$,
$$\left(b_i(a_k+d_k)+b_k(a_i+d_i)\right)/2$$
is an integer.
The term $b_k(a_i+d_i)$ is even, so $b_i(a_k+d_k)$ must
also be even. Since $a_k+d_k$ is odd, $b_i$ is even.

Therefore, if there were two distinct matrices $A_i,A_j$
whose diagonal entries have the same parity, both $b_i$
and $b_j$ would be even. This would make
$b_i e_j+b_j e_i$ even, whereas the first part of the proof
shows that it is odd for every such pair.
Hence there is at most one matrix whose diagonal entries
have the same parity.
\end{proof}

\begin{lemma}\label{lem:scalar-parity}
Let $n$ be an integer.
Let $u_1,\ldots,u_m$ be integers, not necessarily distinct, such that
$u_i u_j+n$ is a square for $i\ne j$.
\begin{enumerate}
\item If $n\equiv2\pmod4$, then $m\le3$.
\item If $n\equiv3\pmod4$, then at most one $u_i$ is even.
\item If $n\equiv1\pmod4$ or $n\equiv4\pmod8$, then at most two $u_i$ are odd.
\item Suppose that $n\equiv5\pmod8$ or $n\equiv8,12\pmod{16}$.
If some $u_i$ is odd, then at most one $u_i$ is even.
\item If $n\equiv8\pmod{16}$ and every $u_i$ is even, then $m\le3$.
\end{enumerate}
\end{lemma}
\begin{proof}
If $n\equiv2$ or $3\pmod4$, two even terms are impossible.
If $n\equiv1$ or $2\pmod4$, the product of any two odd terms
must be congruent to $3\pmod4$, so there are at most two odd terms.
This proves (1), (2) and the first case of (3).
If $n\equiv4\pmod8$, the product of any two odd terms must be
congruent to $5\pmod8$.
Three odd terms $u,v,w$ would give $uv\equiv uw\pmod8$, hence
$v\equiv w\pmod8$ and $vw\equiv1\pmod8$, a contradiction.

For (4), suppose that $u_k$ is odd and $u_i$ is even.
If $n\equiv5\pmod8$, then $u_k u_i+n$ is an odd square,
hence congruent to $1\pmod8$. Thus $u_k u_i\equiv4\pmod8$,
and, since $u_k$ is odd, $u_i\equiv4\pmod8$.
Two even terms would therefore give
$u_i u_j+n\equiv5\pmod8$, which is impossible for a square. 
If $n\equiv8$ or $12\pmod{16}$, then $u_k u_i+n$ is an
even square and hence divisible by $4$.
Since $4\mid n$ and $u_k$ is odd, we obtain $4\mid u_i$.
Two even terms would then give
$u_i u_j+n\equiv n\pmod{16}$, which is not a square modulo $16$.

Finally, in (5), $(u_i/2)(u_j/2)+n/4$ is a square whenever $i\ne j$.
Since $n/4\equiv2\pmod4$, part (1) applies.
\end{proof}

\begin{lemma}\label{lem:odd-upper}
Let $N=(p,r,q)\in U$, where $r$ is odd. Suppose that either \(p\equiv q\equiv1\pmod4\) or \(p\equiv q\equiv4\pmod8\). Then every Jordan \(D(N)\)-set has at most five elements.
\end{lemma}
\begin{proof}
Write \(A_i=(a_i,b_i,d_i)\).
By Lemma~\ref{lem:scalar-parity}(3), at most two of the \(a_i\)
and at most two of the \(d_i\) are odd.

If each matrix has diagonal entries of the same parity,
Lemma~\ref{lem:parity-count} gives at most three matrices. 
Otherwise, the same lemma gives at most one matrix whose diagonal entries have the same parity.
There are at most two matrices with \(a_i\) odd and \(d_i\) even, and at most two with \(a_i\) even and \(d_i\) odd.
Thus there are at most five matrices.
\end{proof}

\begin{lemma}\label{lem:twomodfour-upper}
Let $N=(p,r,q)\in U$, where
$p\equiv q\equiv4\pmod{16}$ and $r\equiv2\pmod4$.  Then every Jordan \(D(N)\)-set has at most five elements.
\end{lemma}
\begin{proof}
Write \(A_i=(a_i,b_i,d_i)\).
We first show that if integers \(u_i\) satisfy
\(u_i u_j+4\equiv0,1,4,9\pmod{16}\) whenever \(i\ne j\),
then at most two of them are not divisible by \(4\).
An odd \(u_i\) and a \(u_j\) congruent to \(2\pmod4\) cannot occur together.
The product of any two odd terms must be congruent to \(5\pmod8\), which is impossible for three terms.
For terms \(u_i=2v_i\) with \(v_i\) odd, the product of any two \(v_i\) must be congruent to \(3\pmod4\), which is also impossible for three terms.

If all \(a_i\) and \(d_i\) are even, put
\[
 B_i=\left(\frac{a_i}{2},b_i,\frac{d_i}{2}\right),
 \qquad
 N'=\left(\frac{p}{4},\frac{r}{2},\frac{q}{4}\right).
\]
The diagonal entries of every square root of \(A_i\circ A_j+N\) are even.
If $(x,y,z)$ is such a square root, then $(x/2,y,z/2)$
is a square root of $B_i\circ B_j+N'$. 
The transformation is injective and preserves nonzeroness,
so the matrices $B_i$ form a Jordan $D(N')$-set.
Since $p/4\equiv q/4\equiv1\pmod4$ and $r/2$ is odd,
Lemma~\ref{lem:odd-upper} gives the bound of five.

Suppose that \(a_k+d_k\) is odd for some \(k\).
If \(a_i+d_i\) is even, the condition \(A_i\circ A_k\in U\) implies that \(b_i\) is even.
By the initial observation, at most two matrices have
$4\nmid a_i$, and at most two have $4\nmid d_i$.
Thus, if there are at least six matrices, at least two
have both diagonal entries divisible by $4$. 
For these two matrices, \(b_i\) and \(b_j\) are even, so the upper-right entry
\(\bigl(b_i(a_j+d_j)+b_j(a_i+d_i)\bigr)/2\) of their Jordan product is divisible by \(4\).
Hence the upper-right entry of \(A_i\circ A_j+N\) is congruent to \(2\pmod4\), while both its diagonal entries are congruent to \(4\pmod{16}\).
Every square root $(x,y,z)$ therefore has $x\equiv z\equiv2\pmod4$.
Its square has upper-right entry $y(x+z)$, which is divisible by $4$,
a contradiction.

It remains to consider the case in which the diagonal entries of each matrix have the same parity and some matrix has two odd diagonal entries.
There are at most two matrices with odd diagonal entries.
Since an odd term and a term congruent to $2\pmod4$
cannot occur together in either diagonal sequence,
every other matrix has both diagonal entries divisible by $4$. 
For the latter matrices put \(e_i=(a_i+d_i)/4\).
For two distinct matrices of this form, the diagonal entries of
\(A_i\circ A_j+N\) are congruent to \(4\pmod{16}\).
As above, its upper-right entry is divisible by \(4\), so
\[ 2(b_i e_j+b_j e_i)+r\equiv0\pmod4,
 \qquad b_i e_j+b_j e_i\equiv1\pmod2.
\]
There are only three nonzero pairs of residues modulo $2$.
Thus, if there were four matrices with both diagonal entries
divisible by $4$, either one would have
$(b_i\bmod2,e_i\bmod2)=(0,0)$, or two distinct matrices would
have the same pair of residues.
In either case, $b_i e_j+b_j e_i$ would be even for some
two distinct matrices $A_i,A_j$, contrary to the preceding
congruence.
Hence there are at most three matrices with both diagonal
entries divisible by $4$.
Together with the at most two matrices with odd diagonal
entries, this gives at most five matrices.
\end{proof}

\begin{theorem}\label{thm:global-five}
Let $N=(p,r,q)\in U$ be a matrix that is not a difference of two
squares in $U$.
Every Jordan $D(N)$-set has at most three elements, except possibly when
$$
p\equiv q\equiv1\pmod8,~~ r\ \text{odd},\qquad
\text{or}\qquad p\equiv q\equiv4\pmod{16},~~ 4\nmid r.
$$
In these two cases every Jordan $D(N)$-set has at most five elements.
The bound of five is best possible.
\end{theorem}
\begin{proof}
By \cite[Theorem~11]{DFsquares}, \(N\) is not a difference of two squares if and only if at least one of the following conditions holds:
$$p\equiv2\pmod4\ \text{or}\ q\equiv2\pmod4,$$
$$p,q\ \text{odd},\quad p\equiv q\pmod4,\quad r\ \text{odd},$$
$$(p,q)\equiv(4,4),(12,12)\pmod{16},\quad 4\nmid r,$$
$$(p,q)\equiv(4,12),(8,8),(12,4)\pmod{16},\quad r\ \text{odd}.$$
Lemmas~\ref{lem:odd-upper} and \ref{lem:twomodfour-upper} give
the bound of five in the two cases stated in the theorem.
The bound of three in the remaining cases was established in
\cite[Propositions~8, 9, 12, 13, 15, 17 and 18]{DFJordan}.
For completeness, we give a unified proof using
Lemmas~\ref{lem:parity-count} and \ref{lem:scalar-parity}.

Write $A_i=(a_i,b_i,d_i)$ for the matrices in the set.
When $r$ is odd and $p\equiv q\pmod2$,
Lemma~\ref{lem:parity-count} shows that it is enough to prove that
at most two matrices have diagonal entries of different parities.
If $p\equiv2$ or $q\equiv2\pmod4$, the bound of three follows directly
from Lemma~\ref{lem:scalar-parity}(1).

Suppose that $p,q,r$ are odd and $p\equiv q\pmod4$.
If $p\equiv q\equiv3\pmod4$, at most one $a_i$ and at most one
$d_i$ are even, by Lemma~\ref{lem:scalar-parity}(2).
Hence at most two matrices have diagonal entries of different parities.
Now suppose that $p\equiv5\pmod8$.
If some $a_i$ is odd, parts (3) and (4) of the same lemma give
at most two odd $a_i$ and at most one even $a_i$, hence at most three matrices.
If every $a_i$ is even, part (3) gives at most two odd $d_i$.
Thus at most two matrices have diagonal entries of different
parities, and the preceding reduction gives the bound of three. 
The case $q\equiv5\pmod8$ follows by interchanging the two diagonal entries.

Suppose that $p\equiv q\equiv12\pmod{16}$ and $4\nmid r$.
If some diagonal entry is odd, Lemma~\ref{lem:scalar-parity}(3),(4)
gives at most three matrices.
If all diagonal entries are even and $r$ is odd,
Lemma~\ref{lem:parity-count} gives the same bound.
If all diagonal entries are even and $r\equiv2\pmod4$, the matrices
$(a_i/2,b_i,d_i/2)$ form a Jordan $D((p/4,r/2,q/4))$-set, as in
the proof of Lemma~\ref{lem:twomodfour-upper}.
The case $p/4\equiv q/4\equiv3\pmod4$, $r/2$ odd, has already been proved.

Finally, let $r$ be odd.
Suppose first that $(p,q)\equiv(4,12)\pmod{16}$.
If some $d_i$ is odd,
Lemma~\ref{lem:scalar-parity}(3),(4) gives at most three matrices.
If every $d_i$ is even, part (3) gives at most two odd $a_i$,
so the bound is again three.
The case $(p,q)\equiv(12,4)\pmod{16}$ is symmetric.
Suppose finally that $p\equiv q\equiv8\pmod{16}$.
If either all $a_i$ or all $d_i$ are even, part (5) gives at most three matrices.
Otherwise, part (4) gives at most one even $a_i$ and at most one
even $d_i$, so at most two matrices have diagonal entries of different parities.

The bound of five is attained by the quintuples in
Propositions~\ref{prop:quintuple-one} and
\ref{prop:quintuple-four} below.
\end{proof}

\begin{proposition}\label{prop:quintuple-one}
For every odd integer $t$, the maximum cardinality of a Jordan
$D((1,t,1))$-set is five.
This maximum is attained by a nondegenerate set.
\end{proposition}
\begin{proof}
Multiplying the upper-right entries of the quintuple in
\cite[Proposition~11]{DFJordan} by $t$, we obtain

$$
(0,0,1),\quad(0,0,3),\quad(0,2t,0),\quad(1,2t,0),\quad(3,2t,0).$$
The matrices are distinct and nonzero.
Denote these matrices by $A_1,\ldots,A_5$, in the order displayed.
Square roots of $A_i\circ A_j+(1,t,1)$, in lexicographic
order of the pairs $(i,j)$ with $i<j$, are
$$
\begin{gathered}
(1,-t,-2),\ (1,t,1),\ (1,t,1),\ (1,t,1),\ (1,2t,1),\\
(1,2t,1),\ (1,2t,1),\ (1,t,1),\ (1,2t,1),\ (2,5t,-1).
\end{gathered}
$$
Since a square in $U$ with two odd diagonal entries has an
even upper-right entry, $(1,t,1)$ is not a square in $U$.
Hence the quintuple is nondegenerate.
The upper bound of five follows from Lemma~\ref{lem:odd-upper}.
\end{proof}

\begin{proposition}\label{prop:quintuple-four}
For every integer $r$ with $4\nmid r$, the maximum cardinality
of a Jordan $D((4,r,4))$-set is five.
This maximum is attained by a nondegenerate set.
\end{proposition}
\begin{proof}
Put $N=(4,r,4)$ and
$$
\begin{gathered}
 A_1=(1,82r,12),\quad
 A_2=(5,6r,96),\quad
 A_3=(12,-4r,5),\\
 A_4=(96,88r,1),\qquad
 A_5=(0,-2r,0).
\end{gathered}$$
Square roots of $A_i\circ A_j+N$, in the same order of pairs
as in the preceding proof, are
$$
\begin{gathered}
(3,113r,34),\ (4,56r,8),\ (10,325r,4),\ (2,-3r,2),\ 
(8,-5r,22),\\  (22,148r,10),\ (2,-25r,2),\
(34,15r,3),\ (2,-4r,2),\ (2,-24r,2).
\end{gathered}$$
For every integer $r$, the first four matrices are distinct
and nonsingular.
Since $4\nmid r$ implies $r\ne0$, all five matrices are
distinct and nonzero. 
A square in $U$ with both diagonal entries equal to $4$
has an upper-right entry divisible by $4$. Therefore the assumption
$4\nmid r$ implies that $N$ is not a square in $U$.
Hence the set is nondegenerate.
The upper bound of five follows from Lemma~\ref{lem:odd-upper}
when $r$ is odd and from Lemma~\ref{lem:twomodfour-upper}
when $r\equiv2\pmod4$.
\end{proof}

\section{Infinite Jordan $D(N)$-sets}

Theorem~\ref{thm:global-five} shows that a Jordan \(D(N)\)-sextuple can exist only if \(N\) is a difference of two squares in \(U\).
We now determine which differences of two squares admit infinite Jordan \(D(N)\)-sets.
The following proposition is based on the construction in \cite[Theorem~5]{DFJordan}.

\begin{proposition}\label{prop:infinite}
Suppose that \(N=K^2-A^2\), where \(K,A\in U\) and \(h=\operatorname{tr}K\ne0\).
Put $E=(0,1,0)$ and $A_t=A+2htE$ for $t\in\mathbb Z$.
Then
\[
 A_s\circ A_t+N
   =\bigl(K+\operatorname{tr}(A)(s+t)E\bigr)^2
       \qquad(s,t\in\mathbb Z).
\]
Consequently, the nonzero matrices among the \(A_t\), \(t\in\mathbb Z\), form an infinite Jordan \(D(N)\)-set.
\end{proposition}
\begin{proof}
The identity follows from \(E^2=0\), \(AE+EA=\operatorname{tr}(A)E\) and \(KE+EK=hE\).
The matrices \(A_t\) are distinct because \(h\ne0\).
\end{proof}

If $K$ has trace zero and a nonzero diagonal entry $k$, then
$K^2=k^2I$. We may therefore replace $K$ by $kI$, whose
trace is $2k\ne0$, without changing the representation. 
Thus Proposition~\ref{prop:infinite} applies whenever there is a representation \(N=K^2-A^2\) in which \(K\) has a nonzero diagonal entry.

If $A$ has diagonal entries $a,d$, then every Jordan product $A_s\circ A_t$ has diagonal entries $a^2,d^2$.
Hence the resulting set is nondegenerate if at least one of $a,d$ is nonzero, and consists of nonsingular matrices if both are nonzero.

\begin{remark}\label{rem:wildner-comparison}
In the archived version cited here, Wildner \cite[Proposition~5.1]{WildnerArchive}
constructs an infinite Jordan $D(N)$-set when $N+a^2I$ is a square
in $U$ for some nonzero integer $a$.
This condition does not hold for $N=(3,r,5)$, where $r\in\mathbb Z$.
Indeed, $3+a^2=x^2$ implies $(x-a)(x+a)=3$ and hence $a^2=1$,
whereas $5+a^2=6$ is not a square.
On the other hand, the matrices
\[
 A=(1,-2r,2),\qquad K=(2,-r,3)
\]
satisfy $K^2-A^2=N$ and $\operatorname{tr}K=5$.
Proposition~\ref{prop:infinite} therefore gives an infinite Jordan
$D(N)$-set consisting of nonsingular matrices.
\end{remark}

\begin{lemma}\label{lem:representations}
Suppose that $N=(p,r,q)\in U$ is a difference of two squares
in $U$. 
There is a representation $N=K^2-A^2$ such that $K$ has a nonzero
diagonal entry whenever $(p,q)\notin\{-1,-4\}^2$, and $A$ has a
nonzero diagonal entry whenever $(p,q)\notin\{1,4\}^2$.
\end{lemma}

\begin{proof}
We use the diagonal choices from
\cite[Propositions~5, 7 and 10]{DFsquares} and verify that they
have the required properties.
Once the diagonal entries of $K$ and $A$ satisfy
$K_{11}^2-A_{11}^2=p$ and $K_{22}^2-A_{22}^2=q$, the remaining
equation is
\[
 \operatorname{tr}(K)K_{12}-\operatorname{tr}(A)A_{12}=r.
\]
For each choice below, the two traces are not both zero.
Thus the upper-right entries can be chosen in $\mathbb Z$ if and
only if $\gcd(\operatorname{tr}K,\operatorname{tr}A)$ divides $r$.

Since $N$ is a difference of two squares, neither $p$ nor $q$
is congruent to $2$ modulo $4$.
We use the following choices, interchanging the diagonal entries
when necessary.
\[
\begin{array}{c|c|c}
(p,q)&\operatorname{diag}K&\operatorname{diag}A\\ \hline
(2k+1,2l+1)&(k+1,l+1)&(k,l)\\
(2k+1,4l)&(k+1,-l-1)&(k,1-l)\\
(4k,4l)&(k+1,l+1)&(k-1,l-1)\\
(16k,16l)&(2k+2,-4l-1)&(2k-2,1-4l)\\
(16k,16l+8)&(2k+2,-4l-3)&(2k-2,-4l-1)
\end{array}
\]
Each row gives the required diagonal differences.
We now check the divisibility condition and specify when each
row is used.

In the first row, the traces are $k+l+2$ and $k+l$, so their
greatest common divisor is $\gcd(k+l,2)$.
If $k+l$ is odd, this is $1$.
If $k+l$ is even, then $p\equiv q\pmod4$.
When $p\equiv q\equiv1\pmod4$, any representation of $N$ as
a difference of two squares has odd diagonal entries in the
first matrix being squared and even diagonal entries in the
second. When $p\equiv q\equiv3\pmod4$, the parities are reversed.
In either case, both squares have even upper-right entries,
so $r$ is even. Hence $\gcd(k+l,2)\mid r$.

In the second row, the traces are $k-l$ and $k-l+1$, so their
greatest common divisor is $1$.
This covers the case in which $p$ is odd and $q$ is divisible
by $4$. The reverse case follows by interchanging the diagonal
entries, since $(a,b,d)\mapsto(d,b,a)$ preserves squares.

Suppose now that $p=4k$ and $q=4l$.
The third row has traces $k+l+2$ and $k+l-2$, whose greatest
common divisor is $\gcd(k+l+2,4)$.
We use this row whenever that divisor divides $r$.
To determine the remaining cases, consider any existing
representation of $N$ as a difference of two squares.
A diagonal difference congruent to $4$ modulo $16$ forces
the corresponding entry of the first matrix to be $2$ modulo
$4$ and that of the second to be $0$ modulo $4$.
For a difference congruent to $12$ modulo $16$, these residues
are reversed. For a difference congruent to $8$ modulo $16$,
both corresponding entries are odd.
Thus the traces in any such representation give
\[
\begin{aligned}
 (p,q)\equiv(4,4),(12,12)\pmod{16}
   &\quad\Longrightarrow\quad4\mid r,\\
 (p,q)\equiv(4,12),(12,4),(8,8)\pmod{16}
   &\quad\Longrightarrow\quad2\mid r.
\end{aligned}
\]
If $k+l$ is odd, the greatest common divisor in the third row
is $1$.
If $k+l\equiv0\pmod4$, it is $2$. The possible pairs
$(k,l)$ modulo $4$ are $(0,0),(1,3),(2,2),(3,1)$, and the
last three give even $r$ by the displayed implications.
If $k+l\equiv2\pmod4$, the divisor is $4$. The possible pairs
are $(0,2),(1,1),(2,0),(3,3)$, and the second and fourth give
$4\mid r$.
Consequently, only
\[
 (p,q)\equiv(0,0),(0,8),(8,0)\pmod{16}
\]
can remain uncovered by the third row.
The last two rows cover the first two classes. Their traces
are odd and differ by $2$, so their greatest common divisor
is $1$. The third class follows by interchanging the diagonal
entries.

In the first three rows, both diagonal entries of $K$ can
vanish only when $(p,q)\in\{-1,-4\}^2$, and both diagonal
entries of $A$ can vanish only when $(p,q)\in\{1,4\}^2$.
In each of the last two rows, the lower-right entries of both
$K$ and $A$ are odd and hence nonzero.
These properties also hold after interchanging the diagonal
entries. Thus the chosen representation has both properties
stated in the lemma.
\end{proof}

We now consider the case $p,q\in\{-1,-4\}$.
Bonciocat, Cipu and Mignotte \cite{BCM} proved that there is no
$D(-1)$-quadruple of positive integers.
There is no $D(-4)$-quadruple of positive integers either.
Indeed, if such a quadruple contained an odd element,
Lemma~\ref{lem:scalar-parity} would give at most two odd elements
and at most one even element.
Thus all four elements would be even, and dividing them by two
would give a $D(-1)$-quadruple.
A set of nonzero integers with either the $D(-1)$-property or the
$D(-4)$-property also has at most three elements. Indeed, any two
elements have positive product, so all elements have the same
sign. Replacing them by their absolute values gives a set of
positive integers with the same property.

\begin{proposition}\label{prop:exceptional-six}
Let $N=(-s^2,r,-t^2)$, where $s,t\in\{1,2\}$, $r\in\mathbb Z$. 
If \(s\ne t\) or \(r\ne0\), every Jordan \(D(N)\)-set has at most six elements. If \(s=t\) and \(r=0\), there is an infinite Jordan \(D(N)\)-set consisting of nonsingular matrices.
\end{proposition}
\begin{proof}
Suppose that $A_1,\ldots,A_7$ are distinct elements of a Jordan
$D(N)$-set, and write $A_i=(a_i,b_i,d_i)$.
For $i\ne j$, the integer $a_i a_j-s^2$ is a square, so the
$a_i$ are nonzero and have the same sign.
There are at most three distinct values among $a_1,\ldots,a_7$,
so at least one value occurs more than once.
If $a_i=a_j=a$ for $i\ne j$, then $a^2-s^2=u^2$ for some integer $u$.
The factorization $(a-u)(a+u)=s^2$ gives $a=\pm s$, since $s\in\{1,2\}$.
Since all the $a_i$ have the same sign, there is a unique repeated
value $a_0\in\{s,-s\}$. Each of the other values occurs once,
and there are at most two of them. Thus at least five of the
seven matrices have upper-left entry $a_0$. Similarly, at least
five have lower-right entry $d_0$ for some $d_0\in\{t,-t\}$.

At least $5+5-7=3$ matrices therefore have diagonal entries
$(a_0,d_0)$. After relabelling, let $A_1,A_2,A_3$ be three of
them. For $1\le i<j\le3$, the square $A_i\circ A_j+N$ has
zero diagonal entries and is therefore the zero matrix.
Its upper-right entry gives
\[
 (a_0+d_0)(b_i+b_j)+2r=0.
\]
The matrices have the same diagonal entries and are distinct,
so $b_1,b_2,b_3$ are distinct.
Subtracting the equations for $(i,j)=(1,2)$ and $(1,3)$ yields
$(a_0+d_0)(b_2-b_3)=0$, hence $a_0+d_0=0$.
Substitution gives $r=0$, and $|a_0|=s$, $|d_0|=t$ give $s=t$.
Thus every Jordan $D(N)$-set has at most six elements when
$s\ne t$ or $r\ne0$.

When \(s=t\) and \(r=0\), the matrices \((s,b,-s)\),
\(b\in\mathbb Z\), form the infinite set constructed by Wildner
\cite[Proposition~4.3]{WildnerArchive}.
Their determinants are \(-s^2\), and their pairwise Jordan products
equal \(s^2I\).
\end{proof}

\begin{theorem}\label{thm:unrestricted-dichotomy}
Let $N=(p,r,q)\in U$. 
There is an infinite Jordan $D(N)$-set if and only if $N$ is a
difference of two squares in $U$, and either
\[
 (p,q)\notin\{-1,-4\}^2
 \qquad\text{or}\qquad p=q,\quad r=0.
\]
If no infinite set exists, every Jordan $D(N)$-set has at most six
elements. This bound is attained. In particular, the existence of
a Jordan $D(N)$-set of cardinality seven implies the existence of
an infinite Jordan $D(N)$-set for the same parameter.
\end{theorem}
\begin{proof}
If $N$ is not a difference of two squares, every Jordan
$D(N)$-set has at most five elements by Theorem~\ref{thm:global-five}.
If $N$ is a difference of two squares and
$(p,q)\notin\{-1,-4\}^2$, Lemma~\ref{lem:representations}
gives a representation $N=K^2-A^2$ in which $K$ has a
nonzero diagonal entry.
If $\operatorname{tr}K=0$, we replace $K$ by a scalar matrix
with the same square and nonzero trace, as explained after
Proposition~\ref{prop:infinite}.
That proposition then gives an infinite set. 
The four remaining pairs $(p,q)\in\{-1,-4\}^2$ are covered by
Proposition~\ref{prop:exceptional-six}.

The sharpness of the bound of six follows from
Proposition~\ref{prop:negative-maxima} below.
\end{proof}

\begin{proposition}\label{prop:negative-maxima}
Let $s\in\{1,2\}$, $r\in\mathbb Z$, and $N=(-s^2,r,-s^2)$. 
The maximum cardinality of a Jordan \(D(N)\)-set is
\[
 \begin{cases}
  \infty,&r=0,\\
  6,&r\ne0\text{ and }2s\mid r,\\
  3,&2s\nmid r.
 \end{cases}
\]
For $N=(-1,r,-4)$ or $N=(-4,r,-1)$, the maximum cardinality
of a Jordan $D(N)$-set is six for every $r\in\mathbb Z$.

All maxima stated in this proposition are also attained by sets
consisting of nonsingular matrices.
Moreover, every Jordan $D(N)$-set for any of these parameters
is nondegenerate.
\end{proposition}
\begin{proof}
Since the diagonal entries of $N$ are negative, $N$ is not a square
in $U$, so every Jordan $D(N)$-set is nondegenerate.
By Proposition~\ref{prop:exceptional-six}, every set has at most
six elements unless \(r=0\) and the diagonal entries of \(N\) are
equal. In that case the same proposition gives an infinite set
of nonsingular matrices.
If $2s\nmid r$, the matrix $N=(-s^2,r,-s^2)$ cannot be a
difference of two squares in $U$.
Indeed, in a representation $N=K^2-A^2$, the equations
$K_{ii}^2-A_{ii}^2=-s^2$, $i=1,2$, force $K_{ii}=0$ and
$A_{ii}=\pm s$, since $s\in\{1,2\}$.
Thus $K^2=0$ and the upper-right entry
$-\operatorname{tr}(A)A_{12}=r$ is divisible by $2s$, a
contradiction.
The diagonal entries of $N$ are congruent to $3\pmod4$ when
$s=1$, and to $12\pmod{16}$ when $s=2$.
Theorem~\ref{thm:global-five} therefore gives at most three
elements. This proves the required upper bounds.

For \(N=(-s^2,r,-s^2)\), the matrices \(A_1,A_2,A_3\) given by
\[
 (1,0,1),\qquad (s^2+1,-r,s^2+1),\qquad
 (s^2+4,-r,s^2)
\]
form a Jordan \(D(N)\)-triple of nonsingular matrices for every \(r\).
Square roots of \(A_i\circ A_j+N\), in the order
\((i,j)=(1,2),(1,3),(2,3)\), are
\[
 (1,0,-1),\qquad (2,0,0),\qquad (s^2+2,-r,s^2).
\]

For \(r\ne0\) with \(2s\mid r\), we obtain sextuples by scaling
the following example for $N_0=(-1,2,-1)$. Take \(A_1,\ldots,A_6\) to be
\[
\begin{gathered}
 (1,2,1),\quad (1,-4,1),\quad (5,0,1),\quad
 (10,11,1),\quad (1,2,5),\quad (1,5,10).
\end{gathered}
\]
Square roots of \(A_i\circ A_j+N_0\), listed in lexicographic order
of the pairs \((i,j)\) with \(i<j\), are
\[
\begin{gathered}
 (0,0,0),\quad (2,4,0),\quad (3,8,0),\quad (0,5,2),\quad (0,6,3),\\
 (-2,5,0),\quad (-3,3,0),\quad (0,-4,2),\quad (0,-5,3),\quad (7,5,0),\\
 (2,2,2),\quad (-2,17,3),\quad (3,46,-2),\quad (3,15,3),\quad (0,4,7).
\end{gathered}
\]
For \(r=2st\) with \(t\in\mathbb Z\setminus\{0\}\), replace every matrix and every listed root
\((a,b,d)\) by \((sa,tb,sd)\). In both squares and Jordan
products this substitution multiplies the diagonal entries by
\(s^2\) and the upper-right entry by \(st\). The resulting six
matrices are therefore distinct and nonsingular and have the Jordan
$D((-s^2,2st,-s^2))$-property. This proves that the maximum is six
when \(r\ne0\) and \(2s\mid r\).

For \(N=(-1,r,-4)\), take
\[
\begin{gathered}
 A_1=(1,r+12,-2),\quad A_2=(1,r-12,-2),\quad A_3=(5,r,-2),\\
 A_4=(10,-8r,-2),\ A_5=(1,r+4,-10),\ A_6=(1,-5r-24,-20).
\end{gathered}
\]
These matrices are distinct and nonsingular for every \(r\).
Square roots of \(A_i\circ A_j+N\), in the same order as above, are
\[
\begin{gathered}
 (0,0,0),\quad (2,r+9,0),\quad (3,3r+16,0),\quad (0,-r-14,4),\\
 (0,-r-17,6),\quad (2,r-9,0),\quad (3,3r-16,0),\quad (0,-r+13,4),\\
 (0,-r+21,6),\quad (7,-r,0),\quad (-2,-r+3,4),\quad (-2,-4r-9,6),\\
 (-3,41r+16,4),\quad (-3,19r-32,6),\quad (0,r+5,14).
\end{gathered}
\]
Finally, the map \((a,b,d)\mapsto(d,b,a)\) preserves squares
and Jordan products. Applying it to the preceding construction
gives a Jordan $D((-4,r,-1))$-set of six nonsingular matrices.
\end{proof}

\section{Nondegenerate sets}

\begin{lemma}\label{lem:positive-diagonal-four}
Let $n\in\{1,4\}$, and let $a_1,\ldots,a_m$ be nonzero integers, not necessarily distinct, 
such that $a_i a_j+n$ is a square for all $i\ne j$.
Then $m\le4$.
\end{lemma}
\begin{proof}
If $a_i=a_j=a$ for some $i\ne j$, then $a^2+n=b^2$ for some integer $b$.
The factorization $(b-a)(b+a)=n$ shows that this is impossible for
$n=1,4$ and $a\ne0$.
Thus the $a_i$ are distinct.
If they all have the same sign, we may assume that they are positive
by replacing every term by its negative if necessary.
The nonexistence of positive $D(1)$-quintuples \cite{HTZ} and
positive $D(4)$-quintuples \cite{BTF} gives $m\le4$.

If both signs occur, the product of two terms with opposite
signs is $-1$ for $n=1$, and either $-3$ or $-4$ for $n=4$.
Fixing a negative term therefore leaves at most two possible
positive terms. Similarly, fixing a positive term leaves at
most two possible negative terms.
Since the terms are distinct, we again obtain $m\le4$.
\end{proof}

\begin{theorem}\label{thm:nondegenerate-dichotomy}
Let $N=(p,r,q)\in U$. 
There exists an infinite nondegenerate Jordan $D(N)$-set if and only
if the following three conditions hold.
\begin{enumerate}
\item $N$ is a difference of two squares in $U$.
\item $(p,q)\notin\{1,4\}^2$.
\item $(p,q)\notin\{-1,-4\}^2$, or $p=q$ and $r=0$.
\end{enumerate}
If at least one of these conditions fails, every nondegenerate
Jordan $D(N)$-set has at most nine elements.
The bound of nine is attained for $N=I$ and $N=4I$.
In particular, if there exists a nondegenerate Jordan $D(N)$-set
with ten elements, then there exists an infinite nondegenerate
Jordan $D(N)$-set.
\end{theorem}
\begin{proof}
Suppose that $(p,q)\in\{1,4\}^2$.
In any Jordan $D(N)$-set, Lemma~\ref{lem:positive-diagonal-four}
allows at most four matrices with a nonzero upper-left entry and
at most four with a nonzero lower-right entry.
Every remaining matrix is strictly upper triangular.
The Jordan product of any two such matrices is zero, so a
nondegenerate set contains at most one of them.
This gives the bound $4+4+1=9$.
If the first condition fails, every Jordan $D(N)$-set has at most
five elements by Theorem~\ref{thm:global-five}.
If the third condition fails, every such set has at most six
elements by Proposition~\ref{prop:exceptional-six}.
Thus the three conditions are necessary, and the stated upper bound holds.

Conversely, suppose that all three conditions hold.
If $(p,q)\notin\{-1,-4\}^2$, Lemma~\ref{lem:representations}
gives $N=K^2-A^2$, where each of $K$ and $A$ has a nonzero diagonal entry.
As observed after Proposition~\ref{prop:infinite}, we may replace
$K$ by a matrix with the same square and nonzero trace.
Proposition~\ref{prop:infinite} then gives an infinite set whose
matrices have the same diagonal entries as $A$.
If these entries are $a,d$, every pairwise Jordan product has
diagonal entries $a^2,d^2$.
Since $(a,d)\ne(0,0)$, this set is nondegenerate.

It remains to consider $N=-s^2I$, where $s\in\{1,2\}$.
The infinite set in Proposition~\ref{prop:exceptional-six} is
nondegenerate, since all its pairwise Jordan products equal $s^2I$.

Proposition~\ref{prop:identity} gives nondegenerate sets of nine
elements for $N=I$ and $N=4I$.
\end{proof}

For $(p,q)\in\{1,4\}^2$, the matrix $N$ is a difference of two
squares in $U$ if and only if it is a square in $U$.
Indeed, if $N=K^2-A^2$, each diagonal equation has the form
$x^2-y^2=1$ or $x^2-y^2=4$, which forces $y=0$.
Both diagonal entries of $A$ are therefore zero, so $A^2=0$
and $N=K^2$. The converse is immediate.
For such a square $N$, the matrices $(0,b,0)$ with
$b\in\mathbb Z\setminus\{0\}$ form an infinite Jordan $D(N)$-set
\cite[Example~2]{DFJordan}. The Jordan product of any two of these matrices is zero.
For the same $N$, however, every nondegenerate Jordan
$D(N)$-set has at most nine elements.

\begin{proposition}\label{prop:identity}
The maximum cardinality of a nondegenerate Jordan $D(N)$-set is nine
for $N=I$, $N=4I$, and for $N=(1,r,4)$ or $N=(4,r,1)$ with
$r\in\mathbb Z$.
\end{proposition}
\begin{proof}
The upper bound of nine follows from
Theorem~\ref{thm:nondegenerate-dichotomy}.
For $N=I$, put $Q=\{1,3,8,120\}$ and take the nine matrices
\[
 \{(a,120,0):a\in Q\}\cup
 \{(0,120,a):a\in Q\}\cup\{(0,120,0)\}.
\]
For distinct $a,b\in Q$, the number $ab+1$ is a square, and
\[
 (a,120,0)\circ(b,120,0)+I
 =\left(\sqrt{ab+1},
   \frac{60(a+b)}{\sqrt{ab+1}+1},1\right)^2.
\]
The middle entry on the right is an integer. If
$\{a,b\}\ne\{8,120\}$, its denominator belongs to
$\{3,4,6,12,20\}$ and divides $60$. For
$\{a,b\}=\{8,120\}$, the middle entry is $240$.
Interchanging the diagonal entries gives square roots for pairs
of the form $(0,120,a),(0,120,b)$.
All other pairs are covered by
\[
\begin{aligned}
 (a,120,0)\circ(0,120,b)+I&=(1,30(a+b),1)^2,\\
 (a,120,0)\circ(0,120,0)+I&=(1,30a,1)^2,\\
 (0,120,a)\circ(0,120,0)+I&=(1,30a,1)^2.
\end{aligned}
\]
The nine matrices are distinct and nonzero. The upper-right entry
of the Jordan product of any two distinct matrices is positive,
so the set is nondegenerate.
Multiplying every matrix and every square root by two gives
a nondegenerate Jordan $D(4I)$-set of nine matrices.

Next let $N=(1,1,4)$ and take
\[
\begin{aligned}
 A_1&=(1,-6,0),  & A_4&=(120,16,0), & A_7&=(0,-2,12),\\
 A_2&=(3,-16,0), & A_5&=(0,16,1),   & A_8&=(0,10,96),\\
 A_3&=(8,-18,0), & A_6&=(0,10,5),   & A_9&=(0,-2,0).
\end{aligned}
\]
These matrices are distinct and nonzero. All their upper-right
entries are even, so their Jordan products have integer entries.
If one matrix is among $A_1,\ldots,A_4$ and the other is among
$A_5,\ldots,A_8$, the diagonal entries of their Jordan product
plus $N$ are $1$ and $4$. The same holds for every pair involving
$A_9$. Thus $A_i\circ A_j+N$ is a square in $U$ for each of
these pairs, since
\[
 (1,u,4)=(-1,u,2)^2\qquad(u\in\mathbb Z).
\]
There remain six pairs among $A_1,\ldots,A_4$ and six among
$A_5,\ldots,A_8$. For these twelve pairs, direct multiplication
gives $A_i\circ A_j+N=K_{ij}^2$, where
\[
\begin{aligned}
 K_{12}&=(2,-4,2),   & K_{24}&=(19,-55,-2), & K_{58}&=(-1,86,10),\\
 K_{13}&=(3,-32,-2), & K_{34}&=(31,-35,-2), & K_{67}&=(-1,8,8),\\
 K_{14}&=(11,-27,2), & K_{56}&=(-1,23,3),   & K_{68}&=(1,22,22),\\
 K_{23}&=(5,-30,-2), & K_{57}&=(-1,32,4),   & K_{78}&=(1,-1,34).
\end{aligned}
\]

Thus the nine matrices have the Jordan $D(N)$-property.

To check nondegeneracy, write $A_i=(a_i,b_i,d_i)$.
The Jordan product of two distinct matrices among
$A_1,\ldots,A_4$ has a nonzero upper-left entry, and that of
two distinct matrices among $A_5,\ldots,A_8$ has a nonzero
lower-right entry.
For $1\le i\le4$ and $5\le j\le8$, the upper-right entry
of $A_i\circ A_j$ is $(a_i b_j+b_i d_j)/2$. It vanishes
if and only if $b_i/a_i=-b_j/d_j$.
The values of these two ratios belong, respectively, to
\[
 \left\{-6,-\frac{16}{3},-\frac94,\frac2{15}\right\}
 \quad\text{and}\quad
 \left\{-16,-2,\frac16,-\frac5{48}\right\},
\]
which are disjoint.
Finally,
\[
 A_i\circ A_9=(0,-(a_i+d_i),0)\ne0\qquad(1\le i\le8),
\]
since these eight matrices have positive trace.
Hence the set is nondegenerate.

For every nonzero integer $r$, the map
$(a,b,d)\mapsto(a,rb,d)$ is injective and preserves squares and
Jordan products. Applying it to the matrices and their square
roots gives a nondegenerate Jordan $D((1,r,4))$-set of nine matrices.
For $r=0$, use instead
\[
 (a,b,d)\longmapsto(a,3b+a-d,d).
\]
This map is given by conjugation, since
\[
 \begin{pmatrix}3&-1\\0&1\end{pmatrix}
 \begin{pmatrix}a&b\\0&d\end{pmatrix}
 \begin{pmatrix}1/3&1/3\\0&1\end{pmatrix}
 =\begin{pmatrix}a&3b+a-d\\0&d\end{pmatrix}.
\]
It maps $U$ into $U$, is injective and preserves squares and
Jordan products. It sends $(1,1,4)$ to $(1,0,4)$.
Applying it to the nine matrices and their square roots therefore
gives the required nondegenerate set for $r=0$.
Finally, interchanging the two diagonal entries of every matrix
and every square root gives the result for $N=(4,r,1)$.
\end{proof}

\begin{proposition}\label{prop:positive-eight}
For $N=(1,r,1)$ with $r$ even and nonzero, and for $N=(4,r,4)$
with $r$ nonzero and divisible by $4$, the maximum cardinality of a
nondegenerate Jordan $D(N)$-set is either eight or nine.
\end{proposition}
\begin{proof}
The upper bound of nine follows from
Theorem~\ref{thm:nondegenerate-dichotomy}.
We first construct a nondegenerate Jordan $D((1,2,1))$-set
of eight matrices. Put $N=(1,2,1)$ and take
\[
\begin{aligned}
 A_1&=(3,8,0),    & A_4&=(1,8,-8),  & A_7&=(0,28,-120),\\
 A_2&=(8,12,0),   & A_5&=(0,-4,-1), & A_8&=(0,-8,0).\\
 A_3&=(120,52,0), & A_6&=(0,20,-3),
\end{aligned}
\]
The matrices are distinct and nonzero.
Since every upper-right entry is divisible by $4$, all Jordan
products have integer entries, and the upper-right entry of
every $A_i\circ A_j+N$ is even.
If one matrix is among $A_1,A_2,A_3$ and the other is among
$A_5,A_6,A_7$, the diagonal entries of $A_i\circ A_j+N$
are both $1$. The same holds for every pair involving $A_8$.
For these pairs, the identity
\[
 (1,2u,1)=(1,u,1)^2\qquad(u\in\mathbb Z)
\]
gives the required square roots.
The remaining pairs consist of the three pairs among
$A_1,A_2,A_3$, the three pairs among $A_5,A_6,A_7$,
and the six pairs of $A_4$ with one of those six matrices.
For these twelve pairs, direct multiplication gives
$A_i\circ A_j+N=K_{ij}^2$, where
\[
\begin{aligned}
 K_{12}&=(5,13,-1),  & K_{24}&=(3,-2,1),  & K_{67}&=(1,-62,19),\\
 K_{13}&=(19,28,1),  & K_{34}&=(11,25,1), & K_{45}&=(1,3,3),\\
 K_{23}&=(31,31,-1), & K_{56}&=(-1,-2,2), & K_{46}&=(-1,-20,5),\\
 K_{14}&=(2,-14,-1), & K_{57}&=(1,19,11), & K_{47}&=(1,-18,31).
\end{aligned}
\]
Thus the eight matrices have the Jordan $D(N)$-property.

To check nondegeneracy, write $A_i=(a_i,b_i,d_i)$.
For each of the twelve pairs listed above, the Jordan product
has a nonzero diagonal entry.
For $1\le i\le3$ and $5\le j\le7$, its upper-right entry
is $(a_i b_j+b_i d_j)/2$ and vanishes if and only if
$b_i/a_i=-b_j/d_j$. The values of these two ratios belong,
respectively, to
\[
 \left\{\frac83,\frac32,\frac{13}{30}\right\}
 \quad\text{and}\quad
 \left\{-4,\frac{20}{3},\frac7{30}\right\},
\]
which are disjoint.
Finally,
\[
 A_i\circ A_8=(0,-4(a_i+d_i),0)\ne0\qquad(1\le i\le7),
\]
since none of these seven matrices has trace zero.
Hence the set is nondegenerate.

For $N=(1,r,1)$ with $r$ even and nonzero, replace every
matrix and every square root $(a,b,d)$ by $(a,(r/2)b,d)$.
For $N=(4,r,4)$ with $r$ nonzero and divisible by $4$,
replace them by $(2a,(r/4)b,2d)$.
In the first substitution, the diagonal entries of each Jordan
product and square are unchanged, and the upper-right entry
is multiplied by $r/2$.
In the second, the diagonal entries are multiplied by $4$,
and the upper-right entry is again multiplied by $r/2$.
Thus the transformed matrices have the required Jordan
$D(N)$-property in both cases.
The divisibility assumptions ensure that the matrices and their
square roots have integer entries. Since $r\ne0$, both
substitutions are injective, and nonzero Jordan products remain
nonzero. We therefore obtain eight distinct nonzero matrices
forming a nondegenerate set in each case.
\end{proof}
It remains to determine, for each of the parameters in the
preceding proposition, whether a nondegenerate Jordan
$D(N)$-set of nine elements exists.

\section{Nonsingular sets}

\begin{proposition}\label{prop:positive-diagonal-four}
Let $N=(p,r,q)\in U$. 
If $p\in\{1,4\}$ or $q\in\{1,4\}$, every Jordan $D(N)$-set
consisting of nonsingular matrices has at most four elements.
If $p=q\in\{1,4\}$, this bound is attained for every integer $r$.
\end{proposition}
\begin{proof}
Both diagonal entries of a nonsingular matrix in $U$ are nonzero.
Applying Lemma~\ref{lem:positive-diagonal-four} to the upper-left
entries if $p\in\{1,4\}$, and to the lower-right entries otherwise,
we obtain the upper bound.
For $p=q=1$, Dujella and Franu\v si\'c
\cite[Proposition~10]{DFJordan} constructed the nonsingular quadruple
\[
 (1,-24r,8),\qquad (3,12r,120),\qquad
 (8,-2r,3),\qquad (120,6r,1).
\]
For $p=q=4$, the first four matrices in
Proposition~\ref{prop:quintuple-four} form a nonsingular quadruple
for every $r$.
In both cases, the pairs of diagonal entries are distinct, so the
matrices are distinct even when $r=0$.
\end{proof}

\begin{theorem}\label{thm:nonsingular-dichotomy}
Let $N=\left(\begin{smallmatrix}p&r\\0&q\end{smallmatrix}\right)\in U$.
There exists an infinite Jordan $D(N)$-set consisting of nonsingular
matrices if and only if the following four conditions hold.
\begin{enumerate}
\item $N$ is a difference of two squares in $U$.
\item $p,q\notin\{1,4\}$.
\item $(p,q)\notin\{-1,-4\}^2$, or $p=q$ and $r=0$.
\item If $r$ is odd and one diagonal entry equals $16$, then the other
is neither equal to $16$ nor congruent to $8$ modulo $16$.
\end{enumerate}
If at least one of these conditions fails, every Jordan $D(N)$-set
consisting of nonsingular matrices has at most six elements, unless
$p=q=16$ and $r$ is odd. In this exceptional case, every such set has
at most nine elements. The bound of six is best possible.
In particular, if there exists a Jordan $D(N)$-set of ten nonsingular
matrices, then there exists an infinite Jordan $D(N)$-set consisting
of nonsingular matrices for the same $N$.
\end{theorem}
\begin{proof}
If the first condition fails, Theorem~\ref{thm:global-five} gives
at most five elements.
If the second condition fails,
Proposition~\ref{prop:positive-diagonal-four} gives at most four
elements. If the third condition fails,
Proposition~\ref{prop:exceptional-six} gives at most six elements.

Suppose that the fourth condition fails.
Interchanging the diagonal entries if necessary, we may assume
that $p=16$, that $r$ is odd, and that either $q=16$ or
$q\equiv8\pmod{16}$.
Write $A_i=(a_i,b_i,d_i)$ for the matrices in the set.
There are at most four matrices with $a_i$ even.
Indeed, for two such matrices,
\[
 (a_i/2)(a_j/2)+4=(a_i a_j+16)/4
\]
is a square. The integers $a_i/2$ are nonzero, so
Lemma~\ref{lem:positive-diagonal-four} applies to this indexed
family.

If every matrix has diagonal entries of the same parity,
Lemma~\ref{lem:parity-count} gives at most three matrices.
Otherwise, that lemma gives at most one matrix whose diagonal
entries have the same parity.
If $q=16$, the preceding argument also gives at most four
matrices with $d_i$ even.
There are therefore at most four matrices with $a_i$ even and
$d_i$ odd, at most four with $a_i$ odd and $d_i$ even, and at
most one whose diagonal entries have the same parity.
This gives at most nine matrices.

If $q\equiv8\pmod{16}$ and all $d_i$ are even,
Lemma~\ref{lem:scalar-parity}(5) gives at most three matrices.
If some $d_i$ is odd, part~(4) of that lemma gives at most one
matrix with $d_i$ even.
There are then at most four matrices with $a_i$ even and $d_i$
odd, at most one with $a_i$ odd and $d_i$ even, and at most one
whose diagonal entries have the same parity.
Thus there are at most six matrices.
This proves necessity of the four conditions and the stated
upper bounds.
For $N=(-1,2,-1)$, Proposition~\ref{prop:negative-maxima} gives
a set of six nonsingular matrices, so the bound of six is
best possible.

Conversely, suppose that all four conditions hold.
If $(p,q)\in\{-1,-4\}^2$, the third condition gives
$N=-s^2I$ for $s\in\{1,2\}$, and
Proposition~\ref{prop:exceptional-six} provides an infinite set
of nonsingular matrices.
We may therefore assume that $(p,q)\notin\{-1,-4\}^2$.
We now construct a representation $N=K^2-A^2$ in which $A$
is nonsingular and $K$ has a nonzero diagonal entry.
We use the diagonal choices from the proof of
Lemma~\ref{lem:representations}, modifying them where necessary
to ensure that both diagonal entries of $A$ are nonzero.

If $p$ and $q$ are odd, use the first row of the table in that
proof. If exactly one is odd, use the second row, interchanging
the diagonal entries when necessary.
For $p=4k$ and $q=4l$, use the third row whenever
$\gcd(k+l+2,4)\mid r$.
The diagonal entries of $A$ in these three rows are
$(k,l)$, $(k,1-l)$ and $(k-1,l-1)$, respectively.
In each case, an entry could vanish only if the corresponding
entry of $N$ were $1$ or $4$, which the second condition excludes.
Thus $A$ is nonsingular. The exclusion of
$(p,q)\in\{-1,-4\}^2$ also ensures that $K$ has a nonzero
diagonal entry.
In particular, $p=q=16$ is covered by the third row, since
$k=l=4$ gives $\gcd(k+l+2,4)=2$, and the fourth condition
implies that $r$ is even.

As shown in the proof of Lemma~\ref{lem:representations},
the only cases that can remain are
\[
 (p,q)\equiv(0,0),(0,8),(8,0)\pmod{16}.
\]
The fourth and fifth rows of the table give $K$ an odd
diagonal entry, so we need only ensure that $A$ is nonsingular.
For $(p,q)=(16k,16l)$, the fourth row gives
$\operatorname{diag}A=(2k-2,1-4l)$.
Both entries are nonzero if $p\ne16$.
If $p=16$, then $q\ne16$, since $p=q=16$ has already been
covered. Apply the fourth row with $p$ and $q$ interchanged,
and then interchange the diagonal entries of both resulting
matrices. The diagonal entries of $A$ are then
$(1-4k,2l-2)$, and both are nonzero.

For $(p,q)=(16k,16l+8)$, the fifth row gives
$\operatorname{diag}A=(2k-2,-4l-1)$, whose entries are both
nonzero if $p\ne16$.
If $p=16$, the fourth condition implies that $r$ is even.
The matrices
\[
 K=(5,-r/2,-2l-3),\qquad A=(3,-r/2,1-2l)
\]
therefore belong to $U$ and satisfy
\[
 K^2-A^2=(16,r,16l+8)=N.
\]
Moreover, $\det A=3(1-2l)\ne0$ because $1-2l$ is odd,
and $K_{11}=5\ne0$.
The case $(p,q)\equiv(8,0)\pmod{16}$ follows by interchanging
the diagonal entries.

For the choices taken from the table, the divisibility checks
in the proof of Lemma~\ref{lem:representations} ensure that
the upper-right entries can be chosen in $\mathbb Z$.
Thus in every case we obtain $N=K^2-A^2$ with $A$ nonsingular
and with a nonzero diagonal entry of $K$.
If $\operatorname{tr}K=0$ and $x\ne0$ is a diagonal entry of
$K$, then $K^2=x^2I=(xI)^2$.
We may therefore replace $K$ by $xI$ without changing the
representation, and the new trace is $2x\ne0$.
Proposition~\ref{prop:infinite} now gives an infinite Jordan
$D(N)$-set. Its matrices differ only in their upper-right
entries, so they all have the same nonzero determinant as $A$
and are therefore nonsingular.
\end{proof}

For $N=(16,r,16)$ with $r$ odd,
Theorem~\ref{thm:nondegenerate-dichotomy} gives an infinite
nondegenerate Jordan $D(N)$-set, whereas every Jordan $D(N)$-set consisting of nonsingular matrices
has at most nine elements.
The next proposition gives lower bounds for nonsingular sets
in the cases $N=(16,r,16)$ and $N=(16,r,8)$.

\begin{proposition}\label{prop:exceptional-nonsingular-lower}
For every odd integer $r$, there is a Jordan $D((16,r,16))$-set
of five nonsingular matrices and a Jordan $D((16,r,8))$-set
of four nonsingular matrices.
\end{proposition}
\begin{proof}
For $N=(16,r,16)$, take $A_1,\ldots,A_5$ to be
\[
\begin{gathered}
(128,309r,3),\quad (528,149r,3),\quad (3,309r,128),\\
(3,149r,528),\quad (3,224r,3).
\end{gathered}
\]
Square roots of $A_i\circ A_j+N$, in lexicographic order
of the pairs $(i,j)$ with $i<j$, are
\[
\begin{gathered}
(260,360r,-5),\ (20,1012r,20),\ (20,1530r,40),\ (20,1040r,-5),\\
(40,1530r,20),\ (40,989r,40),\ (40,1712r,-5),\\
(-5,360r,260),\ (-5,1040r,20),\ (-5,1712r,40).
\end{gathered}
\]
For $N=(16,r,8)$, take $A_1,\ldots,A_4$ to be
\[
(3,-4r,1),\quad (3,2r,28),\quad (16,0,1),\quad (80,-4r,1).
\]
The corresponding square roots, again in lexicographic order
of $(i,j)$ with $i<j$, are
\[
\begin{gathered}
(-5,-57r,6),\ (8,-3r,3),\ (16,-13r,-3),\\
(8,9r,-6),\ (16,2r,-6),\ (-36,r,3).
\end{gathered}
\]
In both constructions, the pairs of diagonal entries are distinct,
and both diagonal entries of every matrix are nonzero.
Thus the matrices are distinct and nonsingular.
\end{proof}

For every odd $r$, the maximum cardinality of a Jordan
$D((16,r,16))$-set of nonsingular matrices lies between five and nine.
For $D((16,r,8))$-sets of nonsingular matrices, it lies between four
and six.
The exact maxima remain to be determined.


\par\medskip
\noindent\textbf{Funding.} The authors were supported by the Croatian Science Foundation under the project no. IP-2022-10-5008 (TEBAG). The authors acknowledge support from the project ``Implementation of cutting-edge research and its application as part of the Scientific Center of Excellence for Quantum and Complex Systems, and Representations of Lie Algebras'', Grant No. PK.1.1.10.0004, co-financed by the European Union through the European Regional Development Fund -- Competitiveness and Cohesion Programme 2021--2027. This research was funded by the European Union's NextGenerationEU through the National Recovery and Resilience Plan 2021--2026 Institutional grant of University of Zagreb Faculty of Science (IK IA 1.1.3. Impact4Math).

\medskip
\noindent\textbf{Use of artificial intelligence.} The authors used ChatGPT 6 (OpenAI) and Claude Opus 5.5 (Anthropic) to assist with developing proofs, constructing examples and counterexamples, searching the literature, writing code for exact computations, and drafting and revising the manuscript. The authors take full responsibility for the content of this paper.

\end{document}